\documentclass[12pt]{amsart}
\usepackage{amssymb}
\usepackage{amsmath}
\usepackage{mathrsfs}
\usepackage{enumerate}
\usepackage{amsthm}
\usepackage{cancel}
\newtheorem{theorem}{Theorem}[section]
\newtheorem{lemma}[theorem]{Lemma}
\newtheorem{corollary}[theorem]{Corollary}

\theoremstyle{definition}

\newtheorem{remark}[theorem]{Remark}

\numberwithin{equation}{section}
\usepackage{color}
\usepackage{hyperref}
\hypersetup{
    colorlinks=true, 
    linkcolor=blue,  
}
\mathchardef\hyphen="2D
\begin{document}
\title{Injective norm of random tensors with independent entries}
\author{March T.~Boedihardjo}
\address{Department of Mathematics, Michigan State University, East Lansing, MI 48824}
\email{boedihar@msu.edu}
\begin{abstract}
We obtain a non-asymptotic bound for the expected injective norm of a random tensor with independent entries. This bound is similar to the bound by Bandeira and van Handel (2016) for the expected spectral norm of a random matrix with independent entries.
\end{abstract}
\maketitle
\section{Introduction}\label{introsection}
Let $[d]=\{1,\ldots,d\}$ and $B_{2}^{d}=\{x\in\mathbb{R}^{d}:\,\|x\|_{2}\leq 1\}$. Let $\{e_{1},\ldots,e_{d}\}$ be the canonical basis for $\mathbb{R}^{d}$. For $d,r\in\mathbb{N}$, the tensor space $(\mathbb{R}^{d})^{\otimes r}=\underbrace{\mathbb{R}^{d}\otimes\ldots\otimes\mathbb{R}^{d}}_{r}$ consists of linear combinations of $x_{1}\otimes\ldots\otimes x_{r}$ where $x_{1},\ldots,x_{r}\in\mathbb{R}^{d}$. The canonical inner product on $(\mathbb{R}^{d})^{\otimes r}$ is defined by
\[\langle x_{1}\otimes\ldots\otimes x_{r},y_{1}\otimes\ldots\otimes y_{r}\rangle:=\langle x_{1},y_{1}\rangle\ldots\langle x_{r},y_{r}\rangle,\]
for $x_{1},\ldots,x_{r},y_{1},\ldots,y_{r}\in\mathbb{R}^{d}$. The injective norm of a given $Z\in(\mathbb{R}^{d})^{\otimes r}$ is defined by
\[\|Z\|_{\mathrm{inj}}:=\sup_{x_{1},\ldots,x_{r}\in B_{2}^{d}}\langle Z,x_{1}\otimes\ldots\otimes x_{r}\rangle.\]
When $r=2$, the injective norm of $Z$ coincides with the spectral norm of $Z$ as a matrix.

Let $g_{i_{1},\ldots,i_{r}}$, for $i_{1},\ldots,i_{r}\in[d]$, be independent standard Gaussian random variables. Let $b_{i_{1},\ldots,i_{r}}\in\mathbb{R}$, for $i_{1},\ldots,i_{r}\in[d]$, be fixed. In this paper, we study the injective norm of the following random tensor
\[Z=\sum_{i_{1},\ldots,i_{r}\in[d]}b_{i_{1},\ldots,i_{r}}g_{i_{1},\ldots,i_{r}}e_{i_{1}}\otimes\ldots\otimes e_{i_{r}}.\]
When $b_{i_{1},\ldots,i_{r}}=1$ for all $i_{1},\ldots,i_{r}\in[d]$, one can use some basic concentration bounds to show that $c\sqrt{d}\leq\mathbb{E}\|Z\|_{\mathrm{inj}}\leq C(r)\sqrt{d}$, where $c>0$ is a universal constant and $C(r)\geq 1$ depends only on $r$. It is known that $C(r)\leq C\sqrt{r\ln r}$ for some universal constant $C\geq 1$ \cite{TomSuz, NDT} and a more precise asymptotic behavior of $\|Z\|_{\mathrm{inj}}$ is described in \cite{DarMcK}.

When $r=2$, matching upper and lower bounds for the spectral norm $\mathbb{E}\|Z\|_{\mathrm{inj}}$ of the random matrix $Z$ are obtained in \cite{LatalaInv}. In particular, a prior result \cite{BvH} shows that
\begin{equation}\label{BvHineq}
\mathbb{E}\|Z\|_{\mathrm{inj}}\leq (1+\epsilon)\left[\max_{i_{1}\in[d]}\left(\sum_{i_{2}\in[d]}b_{i_{1},i_{2}}^{2}\right)^{\frac{1}{2}}+\max_{i_{2}\in[d]}\left(\sum_{i_{1}\in[d]}b_{i_{1},i_{2}}^{2}\right)^{\frac{1}{2}}+\frac{5\sqrt{\ln d}}{\sqrt{\ln(1+\epsilon)}}\max_{i_{1},i_{2}\in[d]}|b_{i_{1},i_{2}}|\right],
\end{equation}
for all $0<\epsilon\leq\frac{1}{2}$. Moreover, $\mathbb{E}\|Z\|_{\mathrm{inj}}$ is always greater than or equal to a universal constant times the sum of the first two terms on the right hand side of (\ref{BvHineq}). Unless the coefficients $b_{i_{1},\ldots,i_{r}}$ are very inhomogeneous, the last term $\sqrt{\ln d}\max_{i_{1},i_{2}}|b_{i_{1},i_{2}}|$ is usually dominated by the sum of the first two terms, and in this case, (\ref{BvHineq}) gives a sharp estimate of $\mathbb{E}\|Z\|_{\mathrm{inj}}$ up to a universal constant factor.

Therefore, it is natural to ask if a version of (\ref{BvHineq}) holds in the context of higher order tensors, i.e., when $r>2$. The proof of (\ref{BvHineq}) in \cite{BvH} uses the moment method technique where one uses the moments of a random matrix to bound its spectral norm. However, this technique is not known to be extendable to tensors in general. There is another proof of (\ref{BvHineq}) (but with different constant factors and with $\sqrt{\ln d}$ being replaced by $\ln d$) which avoids the moment method \cite[Theorem 4.1]{vHTrans}. This proof uses mostly geometric functional analysis techniques, in particular, the Slepian-Fernique inequality. However, it also uses the spectral decomposition of the matrix $(b_{i_{1},i_{2}}^{2})_{i_{1},i_{2}}$ and this is an obstacle for the proof to be extended to tensors.

In another direction, the following inequality is proved in \cite{LatalaRM} (this result was obtained a decade before the inequality (\ref{BvHineq}) was obtained in \cite{BvH})
\begin{equation}\label{Latineq}
\mathbb{E}\|Z\|_{\mathrm{inj}}\leq C\left[\max_{i_{1}\in[d]}\left(\sum_{i_{2}\in[d]}b_{i_{1},i_{2}}^{2}\right)^{\frac{1}{2}}+\max_{i_{2}\in[d]}\left(\sum_{i_{1}\in[d]}b_{i_{1},i_{2}}^{2}\right)^{\frac{1}{2}}+\left(\sum_{i_{1},i_{2}\in[d]}b_{i_{1},i_{2}}^{4}\right)^{\frac{1}{4}}\right].
\end{equation}
Although this bound does not always give a sharp estimate for $\mathbb{E}\|Z\|_{\mathrm{inj}}$ (e.g., when $b_{i_{1},i_{2}}=1$, for $i_{1}=i_{2}$, and $b_{i_{1},i_{2}}=0$ for $i_{1}\neq i_{2}$), the proof of (\ref{Latineq}) in \cite{LatalaRM} has the advantage that it uses only geometric functional analysis techniques without using moments or spectral decompositions, and thus, it has the potential to be extended to tensors. Indeed, recently, a version of (\ref{Latineq}) for tensors is proved in \cite[Theorem 2.1]{Bandeiratensor}, though the argument used in the proof resembles more the one used in \cite{LatalaGC}.

In this paper, we extend (\ref{BvHineq}) to tensors but with a worse constant factor and the $\sqrt{\ln d}$ being replaced by $(\ln d)^{2}$.
\begin{theorem}\label{main}
Let $d,r\in\mathbb{N}$. Suppose that $g_{i_{1},\ldots,i_{r}}$, for $i_{1},\ldots,i_{r}\in[d]$, are independent standard Gaussian random variables and $b_{i_{1},\ldots,i_{r}}\in\mathbb{R}$, for $i_{1},\ldots,i_{r}\in[d]$, are fixed. Let
\[Z=\sum_{i_{1},\ldots,i_{r}\in[d]}b_{i_{1},\ldots,i_{r}}g_{i_{1},\ldots,i_{r}}e_{i_{1}}\otimes\ldots\otimes e_{i_{r}}.\]
Then
\[\mathbb{E}\|Z\|_{\mathrm{inj}}\leq \sqrt{2r}\sum_{k\in[r]}\max_{i_{1},\ldots,i_{k-1},i_{k+1},\ldots,i_{r}\in[d]}\left(\sum_{i_{k}\in[d]}b_{i_{1},\ldots,i_{r}}^{2}\right)^{\frac{1}{2}}+Cr^{3}(\ln d)^{2}\max_{i_{1},\ldots,i_{r}\in[d]}|b_{i_{1},\ldots,i_{r}}|,\]
where $C\geq 1$ is a universal constant.
\end{theorem}
\begin{remark}\label{lowerbound}
We also have the following lower bound
\[(\mathbb{E}\|Z\|_{\mathrm{inj}}^{2})^{\frac{1}{2}}\geq \max_{k\in[r]}\max_{i_{1},\ldots,i_{k-1},i_{k+1},\ldots,i_{r}\in[d]}\left(\sum_{i_{k}\in[d]}b_{i_{1},\ldots,i_{r}}^{2}\right)^{\frac{1}{2}}.\]
Indeed, if we fix $k\in[r]$ and $j_{1},\ldots,j_{k-1},j_{k+1},\ldots,j_{r}\in[d]$, we have
\begin{eqnarray*}
\|Z\|_{\mathrm{inj}}&\geq&\sup_{x\in B_{2}^{d}}\langle Z,e_{j_{1}}\otimes\ldots\otimes e_{j_{k-1}}\otimes x\otimes e_{j_{k+1}}\otimes\ldots\otimes e_{j_{r}}\rangle\\&=&
\sup_{x\in B_{2}^{d}}\sum_{j_{k}\in[d]}b_{j_{1},\ldots,j_{r}}g_{j_{1},\ldots,j_{r}}\langle x,e_{j_{k}}\rangle=\left\|\sum_{j_{k}\in[d]}b_{j_{1},\ldots,j_{r}}g_{j_{1},\ldots,j_{r}}e_{j_{k}}\right\|_{2}.
\end{eqnarray*}
Thus the above lower bound follows.
\end{remark}
\begin{remark}
Since the map $(g_{i_{1},\ldots,i_{r}})_{i_{1},\ldots,i_{r}}\mapsto\|Z\|_{\mathrm{inj}}$ is Lipschitz with respect to the Euclidean norm and the Lipschitz constant is $b:=\max_{i_{1},\ldots,i_{r}\in[d]}|b_{i_{1},\ldots,i_{r}}|$, by Gaussian concentration \cite[Equation (2.35)]{Ledoux}, we have
\[\mathbb{P}(|\|Z\|_{\mathrm{inj}}-\mathbb{E}\|Z\|_{\mathrm{inj}}|\geq t)\leq 2e^{-t^{2}/(2b^{2})},\quad t\geq 0.\]
\end{remark}
\begin{corollary}\label{maincorollary}
Let $d,r\in\mathbb{N}$ and $K>0$. Suppose that $X_{i_{1},\ldots,i_{r}}$, for $i_{1},\ldots,i_{r}\in[d]$, are independent random variables taking values in $[-K,K]$ and have mean $0$. Let
\[X=\sum_{i_{1},\ldots,i_{r}\in[d]}X_{i_{1},\ldots,i_{r}}e_{i_{1}}\otimes\ldots\otimes e_{i_{r}}.\]
Then
\[\mathbb{E}\|X\|_{\mathrm{inj}}\leq 4\sqrt{r}\sum_{k\in[r]}\max_{i_{1},\ldots,i_{k-1},i_{k+1},\ldots,i_{r}\in[d]}\left(\sum_{i_{k}\in[d]}\mathbb{E}X_{i_{1},\ldots,i_{r}}^{2}\right)^{\frac{1}{2}}+Cr^{3}(\ln d)^{2}K,\]
\[(\mathbb{E}\|X\|_{\mathrm{inj}}^{2})^{\frac{1}{2}}\geq \max_{k\in[r]}\max_{i_{1},\ldots,i_{k-1},i_{k+1},\ldots,i_{r}\in[d]}\left(\sum_{i_{k}\in[d]}\mathbb{E}X_{i_{1},\ldots,i_{r}}^{2}\right)^{\frac{1}{2}},\]
and
\[\mathbb{P}(|\|X\|_{\mathrm{inj}}-\mathbb{E}\|X\|_{\mathrm{inj}}|\geq t)\leq Ce^{-ct^{2}/K^{2}},\]
for all $t\geq 0$, where $C,c>0$ are universal constants.
\end{corollary}
\begin{proof}
Let $g_{i_{1},\ldots,i_{r}}$, for $i_{1},\ldots,i_{r}\in[d]$, be independent standard Gaussian random variables that are independent of all the $X_{i_{1},\ldots,i_{r}}$. By Gaussian symmetrization \cite[Lemma 7.4]{vHnotes}, we have
\[\mathbb{E}\|X\|_{\mathrm{inj}}\leq\sqrt{2\pi}\mathbb{E}\left\|\sum_{i_{1},\ldots,i_{r}\in[d]}g_{i_{1},\ldots,i_{r}}X_{i_{1},\ldots,i_{r}}e_{i_{1}}\otimes\ldots\otimes e_{i_{r}}\right\|.\]
So by Theorem \ref{main},
\begin{equation}\label{corollaryeq1}
\mathbb{E}\|X\|_{\mathrm{inj}}\leq \sqrt{2\pi}\left[\sqrt{2r}\sum_{k\in[r]}\mathbb{E}\max_{i_{1},\ldots,i_{k-1},i_{k+1},\ldots,i_{r}\in[d]}\left(\sum_{i_{k}\in[d]}X_{i_{1},\ldots,i_{r}}^{2}\right)^{\frac{1}{2}}+Cr^{3}(\ln d)^{2}K\right].
\end{equation}
By the Rosenthal-type inequality \cite[Theorem 8]{Rosenthal}, for fixed $k\in[r]$ and $i_{1},\ldots,i_{k-1},i_{k+1},\ldots,i_{r}\in[d]$, we have
\[\left[\mathbb{E}\left(\sum_{i_{k}\in[d]}X_{i_{1},\ldots,i_{r}}^{2}\right)^{\lceil 10r\ln d\rceil}\right]^{\frac{1}{\lceil 10r\ln d\rceil}}\leq 1.1\sum_{i_{k}\in[d]}\mathbb{E}X_{i_{1},\ldots,i_{r}}^{2}+8\lceil 10r\ln d\rceil K^{2}.\]
So
\begin{align*}
&\mathbb{E}\max_{i_{1},\ldots,i_{k-1},i_{k+1},\ldots,i_{r}\in[d]}\sum_{i_{k}\in[d]}X_{i_{1},\ldots,i_{r}}^{2}\\\leq&
\left[\sum_{i_{1},\ldots,i_{k-1},i_{k+1},\ldots,i_{r}\in[d]}\mathbb{E}\left(\sum_{i_{k}\in[d]}X_{i_{1},\ldots,i_{r}}^{2}\right)^{\lceil 10r\ln d\rceil}\right]^{\frac{1}{\lceil 10r\ln d\rceil}}\\\leq&
(d^{r-1})^{\frac{1}{\lceil 10r\ln d\rceil}}\left(1.1\max_{i_{1},\ldots,i_{k-1},i_{k+1},\ldots,i_{r}\in[d]}\sum_{i_{k}\in[d]}\mathbb{E}X_{i_{1},\ldots,i_{r}}^{2}+8\lceil 10r\ln d\rceil K^{2}\right).
\end{align*}
Since $(d^{r-1})^{\frac{1}{\lceil 10r\ln d\rceil}}\leq e^{0.1}$, by taking square root above and by (\ref{corollaryeq1}), the upper bound for $\mathbb{E}\|X\|_{\mathrm{inj}}$ follows. The lower bound for $(\mathbb{E}\|X\|_{\mathrm{inj}}^{2})^{\frac{1}{2}}$ follows by applying the argument in Remark \ref{lowerbound}. The concentration bound for $\|X\|_{\mathrm{inj}}$ follows from Talagrand's concentration inequality \cite[Corollary 4.10]{Ledoux} and the facts that $X_{i_{1},\ldots,i_{r}}$ takes values in $[-K,K]$ and the map $(X_{i_{1},\ldots,i_{r}})_{i_{1},\ldots,i_{r}\in[d]}\mapsto\|X\|_{\mathrm{inj}}$ is convex and $1$-Lipschitz with respect to the Euclidean norm.
\end{proof}
The following result follows immediately from Corollary \ref{maincorollary}. This removes the $(\ln d)^{r-2}$ factor in the estimate in \cite{ZZ} at the cost of an extra additive term $Cr^{3}(\ln d)^{2}$.
\begin{corollary}\label{hypergraph}
Let $d,r\in\mathbb{N}$. Suppose that $X_{i_{1},\ldots,i_{r}}$, for $i_{1},\ldots,i_{r}\in[d]$, are independent Bernoulli random variables. Let $p_{i_{1},\ldots,i_{r}}=\mathbb{P}(X_{i_{1},\ldots,i_{r}}=1)$. Let
\[X=\sum_{i_{1},\ldots,i_{r}\in[d]}X_{i_{1},\ldots,i_{r}}e_{i_{1}}\otimes\ldots\otimes e_{i_{r}}.\]
Then
\[\mathbb{E}\|X-\mathbb{E}X\|_{\mathrm{inj}}\leq 4\sqrt{r}\sum_{k\in[r]}\max_{i_{1},\ldots,i_{k-1},i_{k+1},\ldots,i_{r}\in[d]}\left(\sum_{i_{k}\in[d]}(p_{i_{1},\ldots,i_{r}}-p_{i_{1},\ldots,i_{r}}^{2})\right)^{\frac{1}{2}}+Cr^{3}(\ln d)^{2},\]
\[(\mathbb{E}\|X-\mathbb{E}X\|_{\mathrm{inj}}^{2})^{\frac{1}{2}}\geq \max_{k\in[r]}\max_{i_{1},\ldots,i_{k-1},i_{k+1},\ldots,i_{r}\in[d]}\left(\sum_{i_{k}\in[d]}(p_{i_{1},\ldots,i_{r}}-p_{i_{1},\ldots,i_{r}}^{2})\right)^{\frac{1}{2}},\]
and
\[\mathbb{P}(|\|X-\mathbb{E}X\|_{\mathrm{inj}}-\mathbb{E}\|X-\mathbb{E}X\|_{\mathrm{inj}}|\geq t)\leq Ce^{-ct^{2}},\]
for all $t\geq 0$, where $C,c>0$ are universal constants.
\end{corollary}

We end this section with some definitions. A {\it metric} $\rho$ on a set $T$ is a map $\rho:T\times T\to[0,\infty)$ such that $\rho(x,x)=0$, $\rho(x,y)=\rho(y,x)$ and $\rho(x,z)\leq\rho(x,y)+\rho(y,z)$ for all $x,y,z\in T$. The standard condition that $\rho(x,y)=0$ only when $x=y$ is not required. If $(T,\rho)$ is a metric space and $\epsilon>0$, then the {\it covering number} $N(T,\rho,\epsilon)$ is the smallest possible size of an {\it $\epsilon$-cover} of $T$, i.e., a set $S\subset T$ for which every element of $T$ has distance at most $\epsilon$ from an element of $S$.

If $D$ is a $d\times d$ diagonal matrix, then $D_{i,i}$ denotes the $(i,i)$ entry of $D$ and
\[\|D\|_{\infty}:=\max_{i\in[d]}|D_{i,i}|.\]

In the rest of this paper, we prove the main result Theorem \ref{main}. In Section \ref{3section}, we construct 3 technical objects and obtain some properties. In Section \ref{proofsection}, using these properties, we prove Theorem \ref{main}.

\section{Three technical objects}\label{3section}
In this section, we fix $b_{i_{1},\ldots,i_{r}}\in\mathbb{R}$, for $i_{1},\ldots,i_{r}\in[d]$, and we construct 3 technical objects based on these fixed $b_{i_{1},\ldots,i_{r}}$, namely, (1) the multilinear map $\tau$, (2) the $d\times d$ diagonal matrix $D_{x_{1},\ldots,x_{k-1},x_{k+1},\ldots,x_{r}}^{(k)}$ for given $k\in[r]$ and $x_{1},\ldots,x_{k-1},x_{k+1},\ldots,x_{r}\in\mathbb{R}^{d}$, and (3) the metric $\eta^{(k)}$ on $B_{2}^{d}$ for given $k\in[r]$. The properties of these 3 objects are summarized at the end of this section and they will be used to prove Theorem \ref{main} in the next section.
\subsection{Multilinear map and diagonal matrices}
Define the multilinear map\\
$\tau:\underbrace{\mathbb{R}^{d}\times\ldots\times\mathbb{R}^{d}}_{r}\to\mathbb{R}^{d^{r}}$ by
\[\tau(x_{1},\ldots,x_{r}):=(b_{i_{1},\ldots,i_{r}}\langle x_{1},e_{i_{1}}\rangle\ldots\langle x_{r},e_{i_{r}}\rangle)_{i_{1},\ldots,i_{r}\in[d]},\]
for $x_{1},\ldots,x_{r}\in\mathbb{R}^{d}$. Observe that
\begin{equation}\label{taunorm}
\|\tau(x_{1},\ldots,x_{r})\|_{2}=\left(\sum_{i_{1},\ldots,i_{r}\in[d]}b_{i_{1},\ldots,i_{r}}^{2}\langle x_{1},e_{i_{1}}\rangle^{2}\ldots\langle x_{r},e_{i_{r}}\rangle^{2}\right)^{\frac{1}{2}}.
\end{equation}
Thus for every $k\in[r]$, we have
\begin{eqnarray}\label{taud}
\|\tau(x_{1},\ldots,x_{r})\|_{2}&=&
\left(\sum_{i_{k}\in[d]}(D_{x_{1},\ldots,x_{k-1},x_{k+1},\ldots,x_{r}}^{(k)})_{i_{k},i_{k}}^{2}\langle x_{k},e_{i_{k}}\rangle^{2}\right)^{\frac{1}{2}}\\&=&
\|D_{x_{1},\ldots,x_{k-1},x_{k+1},\ldots,x_{r}}^{(k)}x_{k}\|_{2},\nonumber
\end{eqnarray}
where $D_{x_{1},\ldots,x_{k-1},x_{k+1},\ldots,x_{r}}^{(k)}$ is defined to be the $d\times d$ diagonal matrix whose $(i_{k},i_{k})$-entry is given by
\begin{align*}
&(D_{x_{1},\ldots,x_{k-1},x_{k+1},\ldots,x_{r}}^{(k)})_{i_{k},i_{k}}\\:=&
\left(\sum_{i_{1},\ldots,i_{k-1},i_{k+1},\ldots,i_{r}\in[d]}b_{i_{1},\ldots,i_{r}}^{2}\langle x_{1},e_{i_{1}}\rangle^{2}\ldots\langle x_{k-1},e_{i_{k-1}}\rangle^{2}\langle x_{k+1},e_{i_{k+1}}\rangle^{2}\ldots\langle x_{r},e_{i_{r}}\rangle^{2}\right)^{\frac{1}{2}},
\end{align*}
for $i_{k}\in[d]$.

Let us do some examples. When $r=3$ and $k=2$, we have
\begin{eqnarray*}
\|\tau(x_{1},x_{2},x_{3})\|_{2}&=&
\left(\sum_{i_{1},i_{2},i_{3}\in[d]}b_{i_{1},i_{2},i_{3}}^{2}\langle x_{1},e_{i_{1}}\rangle^{2}\langle x_{2},e_{i_{2}}\rangle^{2}\langle x_{3},e_{i_{3}}\rangle^{2}\right)^{\frac{1}{2}}\\&=&
\left(\sum_{i_{2}\in[d]}\left(\sum_{i_{1},i_{3}\in[d]}b_{i_{1},i_{2},i_{3}}^{2}\langle x_{1},e_{i_{1}}\rangle^{2}\langle x_{3},e_{i_{3}}\rangle^{2}\right)\langle x_{2},e_{i_{2}}\rangle^{2}\right)^{\frac{1}{2}}\\&=&
\left(\sum_{i_{2}\in[d]}(D_{x_{1},x_{3}}^{(2)})_{i_{2},i_{2}}^{2}\langle x_{2},e_{i_{2}}\rangle^{2}\right)^{\frac{1}{2}}=\|D_{x_{1},x_{3}}^{(2)}x_{2}\|_{2},
\end{eqnarray*}
for all $x_{1},x_{2},x_{3}\in\mathbb{R}^{d}$, where $D_{x_{1},x_{3}}^{(2)}$ is the $d\times d$ diagonal matrix whose $(i_{2},i_{2})$-entry is given by
\[(D_{x_{1},x_{3}}^{(2)})_{i_{2},i_{2}}=\left(\sum_{i_{1},i_{3}\in[d]}b_{i_{1},i_{2},i_{3}}^{2}\langle x_{1},e_{i_{1}}\rangle^{2}\langle x_{3},e_{i_{3}}\rangle^{2}\right)^{\frac{1}{2}},\]
for $i_{2}\in[d]$. Similarly, we also have $\|\tau(x_{1},x_{2},x_{3})\|_{2}=\|D_{x_{1},x_{2}}^{(3)}x_{3}\|_{2}$ for all $x_{1},x_{2},x_{3}\in\mathbb{R}^{d}$, where $D_{x_{1},x_{2}}^{(3)}$ is the $d\times d$ diagonal matrix with
\[(D_{x_{1},x_{2}}^{(3)})_{i_{3},i_{3}}=\left(\sum_{i_{1},i_{2}\in[d]}b_{i_{1},i_{2},i_{3}}^{2}\langle x_{1},e_{i_{1}}\rangle^{2}\langle x_{2},e_{i_{2}}\rangle^{2}\right)^{\frac{1}{2}},\]
for $i_{3}\in[d]$.
\subsection{The secondary metrics}\label{secondmetric}
For $k\in[r]$, define $\psi_{k}:B_{2}^{d}\to\mathbb{R}^{d^{r-1}}$ by
\[\psi_{k}(x):=\left(\sqrt{\sum_{i_{k}\in[d]}b_{i_{1},\ldots,i_{r}}^{2}\langle x,e_{i_{k}}\rangle^{2}}\right)_{i_{1},\ldots,i_{k-1},i_{k+1},\ldots,i_{r}\in[d]},\]
for $x\in B_{2}^{d}$. For $k\in[r]$, define the metric $\eta^{(k)}$ on $B_{2}^{d}$ by
\[\eta^{(k)}(x,y):=\|\psi_{k}(x)-\psi_{k}(y)\|_{\infty},\]
for $x,y\in B_{2}^{d}$.
\pagebreak

The following result establishes a relation between the multilinear map $\tau$, introduced in the previous subsection, and the metrics $\eta^{(1)},\ldots,\eta^{(r)}$.
\begin{lemma}\label{taueta}
For all $x_{1},\ldots,x_{r},y_{1},\ldots,y_{r}\in B_{2}^{d}$, we have
\[\|\tau(x_{1},\ldots,x_{r})\|_{2}-\|\tau(y_{1},\ldots,y_{r})\|_{2}\leq\sum_{k=1}^{r}\eta^{(k)}(x_{k},y_{k}).\]
\end{lemma}
\begin{proof}
By (\ref{taunorm}), for all $k\in[r]$ and $x_{1},\ldots,x_{r},y_{k}\in B_{2}^{d}$, we have
\begin{align*}
&\|\tau(x_{1},\ldots,x_{r})\|_{2}-\|\tau(x_{1},\ldots,x_{k-1},y_{k},x_{k+1},\ldots,x_{r})\|_{2}\\=&
\left(\sum_{i_{1},\ldots,i_{r}}b_{i_{1},\ldots,i_{r}}^{2}\langle x_{1},e_{i_{1}}\rangle^{2}\ldots\langle x_{r},e_{i_{r}}\rangle^{2}\right)^{\frac{1}{2}}\\&
-\left(\sum_{i_{1},\ldots,i_{r}}b_{i_{1},\ldots,i_{r}}^{2}\langle x_{1},e_{i_{1}}\rangle^{2}\ldots\langle x_{k-1},e_{i_{k-1}}\rangle^{2}\langle y_{k},e_{i_{k}}\rangle^{2}\langle x_{k+1},e_{i_{k+1}}\rangle^{2}\ldots\langle x_{r},e_{i_{r}}\rangle^{2}\right)^{\frac{1}{2}}\\=&
\left(\sum_{i_{1},\ldots,i_{k-1},i_{k+1},\ldots,i_{r}}[\psi_{k}(x_{k})]_{i_{1},\ldots,i_{k-1},i_{k+1},\ldots,i_{r}}^{2}\prod_{j\in[r]\backslash\{k\}}\langle x_{j},e_{i_{j}}\rangle^{2}\right)^{\frac{1}{2}}\\&
-\left(\sum_{i_{1},\ldots,i_{k-1},i_{k+1},\ldots,i_{r}}[\psi_{k}(y_{k})]_{i_{1},\ldots,i_{k-1},i_{k+1},\ldots,i_{r}}^{2}\prod_{j\in[r]\backslash\{k\}}\langle x_{j},e_{i_{j}}\rangle^{2}\right)^{\frac{1}{2}}\\\leq&
\left(\sum_{i_{1},\ldots,i_{k-1},i_{k+1},\ldots,i_{r}}[\psi_{k}(x_{k})-\psi_{k}(y_{k})]_{i_{1},\ldots,i_{k-1},i_{k+1},\ldots,i_{r}}^{2}\prod_{j\in[r]\backslash\{k\}}\langle x_{j},e_{i_{j}}\rangle^{2}\right)^{\frac{1}{2}}\\\leq&
\|\psi_{k}(x_{k})-\psi_{k}(y_{k})\|_{\infty}\left(\sum_{i_{1},\ldots,i_{k-1},i_{k+1},\ldots,i_{r}}\prod_{j\in[r]\backslash\{k\}}\langle x_{j},e_{i_{j}}\rangle^{2}\right)^{\frac{1}{2}}\\=&
\|\psi_{k}(x_{k})-\psi_{k}(\widetilde{x}_{k})\|_{\infty}
\prod_{j\in[r]\backslash\{k\}}\|x_{j}\|_{2}\leq\|\psi_{k}(x_{k})-\psi_{k}(y_{k})\|_{\infty}=\eta^{(k)}(x_{k},y_{k}).
\end{align*}
Therefore, for all $k\in[r]$ and $x_{1},\ldots,x_{r},y_{k}\in B_{2}^{d}$, we have
\[\|\tau(x_{1},\ldots,x_{k})\|_{2}-\|\tau(x_{1},\ldots,x_{k-1},y_{k},x_{k+1},\ldots,x_{r})\|_{2}\leq\eta^{(k)}(x_{k},y_{k}).\]
So for all $x_{1},\ldots,x_{r},y_{1},\ldots,y_{r}\in B_{2}^{d}$, we have
\[\|\tau(x_{1},\ldots,x_{r})\|_{2}-\|\tau(y_{1},\ldots,y_{r})\|_{2}\leq\sum_{k=1}^{r}\eta^{(k)}(x_{k},y_{k}).\]
\end{proof}
\pagebreak

The following result establishes a relation between the diagonal matrix $D_{x_{1},\ldots,x_{k-1},x_{k+1},\ldots,x_{r}}^{(k)}$, introduced in the previous subsection, and the metrics $\eta^{(1)},\ldots,\eta^{(r)}$.
\begin{lemma}\label{deta}
For all $k\in[r]$, $x_{1},\ldots,x_{k-1},x_{k+1},\ldots,x_{r},y_{1},\ldots,y_{k-1},y_{k+1},\ldots,y_{r}\in B_{2}^{d}$, we have
\begin{align*}
&\|D_{x_{1},\ldots,x_{k-1},x_{k+1},\ldots,x_{r}}^{(k)}-D_{y_{1},\ldots,y_{k-1},y_{k+1},\ldots,y_{r}}^{(k)}\|_{\infty}\\\leq&
\eta^{(1)}(x_{1},y_{1})+\ldots+\eta^{(k-1)}(x_{k-1},y_{k-1})+\eta^{(k+1)}(x_{k+1},y_{k+1})+\ldots+\eta^{(r)}(x_{r},y_{r}).
\end{align*}
\end{lemma}
\begin{proof}
Since $D_{x_{1},\ldots,x_{k-1},x_{k+1},\ldots,x_{r}}^{(k)}$ and $D_{y_{1},\ldots,y_{k-1},y_{k+1},\ldots,y_{r}}^{(k)}$ are diagonal matrices with nonnegative entries,
\begin{align*}
&\|D_{x_{1},\ldots,x_{k-1},x_{k+1},\ldots,x_{r}}^{(k)}-D_{y_{1},\ldots,y_{k-1},y_{k+1},\ldots,y_{r}}^{(k)}\|_{\infty}\\=&
\max_{i\in[d]}\left|\|D_{x_{1},\ldots,x_{k-1},x_{k+1},\ldots,x_{r}}^{(k)}e_{i}\|_{2}-\|D_{y_{1},\ldots,y_{k-1},y_{k+1},\ldots,y_{r}}^{(k)}e_{i}\|_{2}\right|\\=&
\max_{i\in[d]}|\|\tau(x_{1},\ldots,x_{k-1},e_{i},x_{k+1},\ldots,x_{r})\|_{2}-\|\tau(y_{1},\ldots,y_{k-1},e_{i},y_{k+1},\ldots,y_{r})\|_{2}|\\&
\text{by }(\ref{taud})\\\leq&
\eta^{(1)}(x_{1},y_{1})+\ldots+\eta^{(k-1)}(x_{k-1},y_{k-1})+\eta^{(k+1)}(x_{k+1},y_{k+1})+\ldots+\eta^{(r)}(x_{r},y_{r})\\&
\text{by Lemma }\ref{taueta}.
\end{align*}
\end{proof}
In the rest of this subsection, we bound the covering number $N(B_{2}^{d},\eta^{(k)},\epsilon)$ and the Dudley's entropy integral for $\eta^{(k)}$.
\begin{lemma}\label{abt0}
Let $\alpha,\beta,t_{0}\geq 0$. If $|\sqrt{\alpha}-\sqrt{\beta}|\geq t_{0}$ then $|\alpha-\beta|\geq\max(t_{0}\sqrt{\beta},t_{0}^{2})$.
\end{lemma}
\begin{proof}
We have
\[|\alpha-\beta|=|\sqrt{\alpha}-\sqrt{\beta}|\cdot|\sqrt{\alpha}+\sqrt{\beta}|\geq t_{0}(\sqrt{\alpha}+\sqrt{\beta}).\]
Since $\sqrt{\alpha}+\sqrt{\beta}\geq\sqrt{\beta}$ and $\sqrt{\alpha}+\sqrt{\beta}\geq|\sqrt{\alpha}-\sqrt{\beta}|\geq t_{0}$, the result follows.
\end{proof}
\begin{lemma}\label{sqrtconc}
Let $Z_{1},\ldots,Z_{n}$ be independent, identically distributed random variables taking values in $[0,1]$. Then
\[\mathbb{P}\left(\left|\sqrt{\frac{Z_{1}+\ldots+Z_{n}}{n}}-\sqrt{\mathbb{E}Z_{1}}\right|\geq\frac{t}{\sqrt{n}}\right)\geq 2e^{-t^{2}/4},\]
for all $t\geq 0$.
\end{lemma}
\begin{proof}
We use the argument in \cite[Proof of Theorem 3.1.1]{Romanbook}. By Bernstein's inequality \cite[Theorem 2.8.4]{Romanbook}, we have
\[\mathbb{P}(|Z_{1}+\ldots+Z_{n}-n\mathbb{E}Z_{1}|\geq u)\leq 2\exp\left(-\frac{u^{2}/2}{n\mathbb{E}(Z_{1}-\mathbb{E}Z_{1})^{2}+u/3}\right),\]
for all $u\geq 0$. Since $\mathbb{E}(Z_{1}-\mathbb{E}Z_{1})^{2}\leq\mathbb{E}Z_{1}^{2}\leq\mathbb{E}Z_{1}$, we have
\[\frac{u^{2}/2}{n\mathbb{E}(Z_{1}-\mathbb{E}Z_{1})^{2}+u/3}\geq\frac{u^{2}/2}{n\mathbb{E}Z_{1}+u/3}\geq\frac{u^{2}/2}{2\max(n\mathbb{E}Z_{1},u/3)}=\min\left(\frac{u^{2}}{4n\mathbb{E}Z_{1}},\,\frac{3u}{4}\right),\]
and so
\[\mathbb{P}(|Z_{1}+\ldots+Z_{n}-n\mathbb{E}Z_{1}|\geq u)\leq 2\exp\left(-\min\left(\frac{u^{2}}{4n\mathbb{E}Z_{1}},\,\frac{3u}{4}\right)\right),\]
for all $u\geq 0$. Thus,
\begin{equation}\label{sqrtconceq1}
\mathbb{P}\left(\left|\frac{Z_{1}+\ldots+Z_{n}}{n}-\mathbb{E}Z_{1}\right|\geq u\right)\leq
2\exp\left(-\min\left(\frac{u^{2}n}{4\mathbb{E}Z_{1}},\,\frac{3un}{4}\right)\right),
\end{equation}
for all $u\geq 0$. By Lemma \ref{abt0},
\begin{align*}
&\mathbb{P}\left(\left|\sqrt{\frac{Z_{1}+\ldots+Z_{n}}{n}}-\sqrt{\mathbb{E}Z_{1}}\right|\geq\frac{t}{\sqrt{n}}\right)\\\leq&
\mathbb{P}\left(\left|\frac{Z_{1}+\ldots+Z_{n}}{n}-\mathbb{E}Z_{1}\right|\geq\max\left(\frac{t\sqrt{\mathbb{E}Z_{1}}}{\sqrt{n}},\,\frac{t^{2}}{n}\right)\right).
\end{align*}
Taking $u=\max(\frac{t\sqrt{\mathbb{E}Z_{1}}}{\sqrt{n}},\frac{t^{2}}{n})$ in (\ref{sqrtconceq1}), we have $\min(\frac{u^{2}n}{4\mathbb{E}Z_{1}},\frac{3un}{4})\geq\frac{t^{2}}{4}$ and the result follows.
\end{proof}
In the sequel, if $z$ is a vector with nonnegative entries, then $\sqrt{z}$ is the vector obtained by taking square root on each entry of $z$.
\begin{lemma}\label{coveringmaurey}
Let $d_{0}\geq 2$ and let $S\subset[0,1]^{d_{0}}$ be a finite set. Then the covering number
\[N(\{\sqrt{z}:z\in\mathrm{conv}(S)\},\|\,\|_{\infty},\epsilon)\leq\exp\left(\frac{C(\ln d_{0})(\ln|S|)}{\epsilon^{2}}\right).\]
for all $\epsilon>0$, where $C\geq 1$ is a universal constant and $\mathrm{conv}(S)$ is the convex hull of $S$.
\end{lemma}
\begin{proof}
We use the Maurey's empirical method. Fix $z\in\mathrm{conv}(S)$. Write $z=\sum_{x\in S}a_{x}x$ where $\sum_{x\in S}a_{x}=1$ with $a_{x}\geq 0$ for all $x\in S$. Randomly and independently select $z_{1},\ldots,z_{n}\in S$ according to the probability mass $(a_{x})_{x\in S}$. Then by Lemma \ref{sqrtconc},
\[\mathbb{P}\left(\left|\sqrt{\left\langle\frac{z_{1}+\ldots+z_{n}}{n},e_{i}\right\rangle}-\sqrt{\langle z,e_{i}\rangle}\right|\geq\frac{t}{\sqrt{n}}\right)\leq 2e^{-t^{2}/4},\]
for all $t\geq 0$ and $i\in[d_{0}]$, so
\begin{align*}
&\mathbb{P}\left(\left\|\sqrt{\frac{z_{1}+\ldots+z_{n}}{n}}-\sqrt{z}\right\|_{\infty}\geq\frac{t}{\sqrt{n}}\right)\\=&
\mathbb{P}\left(\bigcup_{i\in[d_{0}]}\left(\,\left|\sqrt{\left\langle\frac{z_{1}+\ldots+z_{n}}{n},e_{i}\right\rangle}-\sqrt{\langle z,e_{i}\rangle}\right|\leq\frac{t}{\sqrt{n}}\right)\right)\leq 2d_{0}e^{-t^{2}/4},
\end{align*}
for all $t\geq 0$. Take $t=\sqrt{4\ln(4d_{0})}$. Then $2d_{0}e^{-t^{2}/4}=\frac{1}{2}<1$. Thus we conclude that for every $z\in\mathrm{conv}(S)$ and every $n\in\mathbb{N}$, there exist $z_{1},\ldots,z_{n}\in S$ such that 
\[\left\|\sqrt{\frac{z_{1}+\ldots+z_{n}}{n}}-\sqrt{z}\right\|_{\infty}\leq\sqrt{\frac{4\ln(4d_{0})}{n}}.\]
Fix $\epsilon>0$. Take $n=\lceil\frac{4\ln(4d_{0})}{\epsilon^{2}}\rceil$. Then $\{\sqrt{\frac{z_{1}+\ldots+z_{n}}{n}}:\,z_{1},\ldots,z_{n}\in S\}$ is an $\epsilon$-cover of $\{\sqrt{z}:\,z\in\mathrm{conv}(S)\}$ with cardinality at most $|S|^{n}=e^{n\ln|S|}$.
\end{proof}
Recall from the beginning of this subsection that $\psi_{k}:B_{2}^{d}\to\mathbb{R}^{d^{r-1}}$ is defined by\[\psi_{k}(x):=\left(\sqrt{\sum_{i_{k}\in[d]}b_{i_{1},\ldots,i_{r}}^{2}\langle x,e_{i_{k}}\rangle^{2}}\right)_{i_{1},\ldots,i_{k-1},i_{k+1},\ldots,i_{r}\in[d]},\]
for $x\in B_{2}^{d}$, and the metric $\eta^{(k)}$ on $B_{2}^{d}$ is defined by
\[\eta^{(k)}(x,y):=\|\psi_{k}(x)-\psi_{k}(y)\|_{\infty},\]
for $x,y\in B_{2}^{d}$ and $k\in[r]$.
\begin{lemma}\label{coveretak}
We have
\[\int_{0}^{\infty}\max_{k\in[r]}\sqrt{\ln N(B_{2}^{d},\eta^{(k)},\epsilon)}\,d\epsilon\leq C\sqrt{r}(\ln d)^{2}b,\]
where $C\geq 1$ is a universal constant and $b=\max_{i_{1},\ldots,i_{r}\in[d]}|b_{i_{1},\ldots,i_{r}}|$.
\end{lemma}
\begin{proof}
Since $\eta^{(k)}(x,y)=\|\psi_{k}(x)-\psi_{k}(y)\|_{\infty}$ for all $x,y\in B_{2}^{d}$,
\[N(B_{2}^{d},\eta^{(k)},\epsilon)=N(\mathrm{range}(\psi_{k}),\|\,\|_{\infty},\epsilon).\]
Observe that $\mathrm{range}(\psi_{k})=\{\sqrt{z}:\,z\in\mathrm{conv}(S)\}$,
where
\[S=\{(b_{i_{1},\ldots,i_{r}}^{2})_{i_{1},\ldots,i_{k-1},i_{k+1},\ldots,i_{r}\in[d]}:\,i_{k}\in[d]\}\cup\{0\}\subset\mathbb{R}^{d^{r-1}}.\]
Thus,
\begin{eqnarray*}
N(B_{2}^{d},\eta^{(k)},\epsilon)&=&N(\{\sqrt{z}:\,z\in\mathrm{conv}(S)\},\|\,\|_{\infty},\epsilon)\\&=&
N\left(\left\{\frac{1}{b}\sqrt{z}:\,z\in\mathrm{conv}(S)\right\},\|\,\|_{\infty},\frac{\epsilon}{b}\right)\\&=&
N\left(\left\{\sqrt{z}:\,z\in\mathrm{conv}\left(\frac{1}{b^{2}}S\right)\right\},\|\,\|_{\infty},\frac{\epsilon}{b}\right).
\end{eqnarray*}
Since $|S|\leq d+1$ and $\frac{1}{b^{2}}S\subset[0,1]^{d^{r-1}}$, by Lemma \ref{coveringmaurey} with $d_{0}=d^{r-1}$, we have
\[N\left(\left\{\sqrt{z}:\,z\in\mathrm{conv}\left(\frac{1}{b^{2}}S\right)\right\},\|\,\|_{\infty},\frac{\epsilon}{b}\right)\leq
\exp\left(\frac{C(\ln d^{r-1})(\ln(d+1))}{(\epsilon/b)^{2}}\right).\]
Therefore,
\begin{equation}\label{coveretakeq1}
N(B_{2}^{d},\eta^{(k)},\epsilon)\leq\exp\left(\frac{2Cr(b\ln d)^{2}}{\epsilon^{2}}\right),
\end{equation}
for all $\epsilon>0$ and $k\in[r]$. On the other hand, since
\begin{eqnarray*}
\eta^{(k)}(x,y)&=&\max_{i_{1},\ldots,i_{k-1},i_{k+1},\ldots,i_{r}\in[d]}\left|\sqrt{\sum_{i_{k}\in[d]}b_{i_{1},\ldots,i_{r}}^{2}\langle x,e_{i_{k}}\rangle^{2}}-\sqrt{\sum_{i_{k}\in[d]}b_{i_{1},\ldots,i_{r}}^{2}\langle y,e_{i_{k}}\rangle^{2}}\right|\\&\leq&
\max_{i_{1},\ldots,i_{k-1},i_{k+1},\ldots,i_{r}\in[d]}\left|\sqrt{\sum_{i_{k}\in[d]}b_{i_{1},\ldots,i_{r}}^{2}\langle x-y,e_{i_{k}}\rangle^{2}}\right|
\leq b\|x-y\|_{2},
\end{eqnarray*}
for all $x,y\in B_{2}^{d}$, we have
\begin{equation}\label{coveretakeq2}
N(B_{2}^{d},\eta^{(k)},\epsilon)\leq N(B_{2}^{d},b\|\,\|_{2},\epsilon)=N\left(B_{2}^{d},\|\,\|_{2},\frac{\epsilon}{b}\right)\leq\left(\frac{3b}{\epsilon}\right)^{d},
\end{equation}
for all $0<\epsilon\leq b$ and $k\in[r]$, and the metric space $(B_{2}^{d},\eta^{(k)})$ has diameter at most $2b$. The above two upper bounds (\ref{coveretakeq1}) and (\ref{coveretakeq2}) for $N(B_{2}^{d},\eta^{(k)},\epsilon)$ give
\begin{align*}
&\int_{0}^{\infty}\max_{k\in[r]}\sqrt{\ln N(B_{2}^{d},\eta^{(k)},\epsilon)}\,d\epsilon\\=&
\int_{\frac{b}{d}}^{2b}\max_{k\in[r]}\sqrt{\ln N(B_{2}^{d},\eta^{(k)},\epsilon)}\,d\epsilon+\int_{0}^{\frac{b}{d}}\max_{k\in[r]}\sqrt{\ln N(B_{2}^{d},\eta^{(k)},\epsilon)}\,d\epsilon\\\leq&
\int_{\frac{b}{d}}^{2b}\frac{\sqrt{2Cr}(b\ln d)}{\epsilon}\,d\epsilon+\int_{0}^{\frac{b}{d}}\sqrt{d\ln\left(\frac{3b}{\epsilon}\right)}\,d\epsilon\\=&
\sqrt{2Cr}(b\ln d)\ln(2d)+\frac{b}{\sqrt{d}}\int_{0}^{1}\sqrt{\ln\left(\frac{3d}{\epsilon}\right)}\,d\epsilon.
\end{align*}
Since $\int_{0}^{1}\sqrt{\ln(\frac{3}{\epsilon})}\,d\epsilon<\infty$, the result follows.
\end{proof}
\subsection{Summary of properties}\label{summary}
\begin{itemize}
\item ({\bf Definition of $\tau$}) The multilinear map
$\tau:\underbrace{\mathbb{R}^{d}\times\ldots\times\mathbb{R}^{d}}_{r}\to\mathbb{R}^{d^{r}}$ is defined by
\begin{equation}\label{staudef}
\tau(x_{1},\ldots,x_{r}):=(b_{i_{1},\ldots,i_{r}}\langle x_{1},e_{i_{1}}\rangle\ldots\langle x_{r},e_{i_{r}}\rangle)_{i_{1},\ldots,i_{r}\in[d]},
\end{equation}
for $x_{1},\ldots,x_{r}\in\mathbb{R}^{d}$.\\
\item ({\bf Relation between $\tau$ and $D_{x_{1},\ldots,x_{k-1},x_{k+1},\ldots,x_{r}}^{(k)}$})\\
For given $k\in[r]$ and $x_{1},\ldots,x_{k-1},x_{k+1},\ldots,x_{r}\in\mathbb{R}^{d}$, the $d\times d$ diagonal matrix $D_{x_{1},\ldots,x_{k-1},x_{k+1},\ldots,x_{r}}^{(k)}$ satisfies
\begin{equation}\label{staud}
\|\tau(x_{1},\ldots,x_{r})\|_{2}=
\|D_{x_{1},\ldots,x_{k-1},x_{k+1},\ldots,x_{r}}^{(k)}x_{k}\|_{2},
\end{equation}
for all $x_{k}\in\mathbb{R}^{d}$.\\
\item ({\bf Relation between $D_{x_{1},\ldots,x_{k-1},x_{k+1},\ldots,x_{r}}^{(k)}$ and $\eta^{(1)},\ldots,\eta^{(r)}$}) The metrics $\eta^{(1)},\ldots,\eta^{(r)}$ on $B_{2}^{d}$ satisfy
\begin{align}\label{sdeta}
&\|D_{x_{1},\ldots,x_{k-1},x_{k+1},\ldots,x_{r}}^{(k)}-D_{y_{1},\ldots,y_{k-1},y_{k+1},\ldots,y_{r}}^{(k)}\|_{\infty}\\\leq&
\eta^{(1)}(x_{1},y_{1})+\ldots+\eta^{(k-1)}(x_{k-1},y_{k-1})+\eta^{(k+1)}(x_{k+1},y_{k+1})+\ldots+\eta^{(r)}(x_{r},y_{r}),\nonumber
\end{align}
for all $k\in[r]$, $x_{1},\ldots,x_{k-1},x_{k+1},\ldots,x_{r},y_{1},\ldots,y_{k-1},y_{k+1},\ldots,y_{r}\in B_{2}^{d}$.\\
\item ({\bf Dudley's entropy integral for $\eta^{(k)}$})
\begin{equation}\label{setadudley}
\int_{0}^{\infty}\max_{k\in[r]}\sqrt{\ln N(B_{2}^{d},\eta^{(k)},\epsilon)}\,d\epsilon\leq C\sqrt{r}(\ln d)^{2}b,
\end{equation}
where $b=\max_{i_{1},\ldots,i_{r}\in[d]}|b_{i_{1},\ldots,i_{r}}|$.
\end{itemize}

\section{Proof of the main result}\label{proofsection}
The proof of Theorem \ref{main} below is partly based on the argument in \cite[Proof of Theorem 4.1]{vHTrans} that uses the Slepian-Fernique inequality to bound the expected spectral norm of a Gaussian random matrix. The idea of the argument there is as follows. The induced metric from the spectral norm of the Gaussian random matrix is shown to be bounded by the induced metric of another Gaussian process plus a second term. Using a spectral decomposition, one can bound this second term by the induced metric of some Gaussian process. Thus the Slepian-Fernique inequality can be applied.

In the proof of Theorem \ref{main} below, we are also able to show that the induced metric from the injective norm of the Gaussian tensor $Z$ defined in Theorem \ref{main} is bounded by the induced metric of another Gaussian process plus a second metric. However, it is not clear if we can bound this second metric by the induced metric of a useful Gaussian process. Thus we need a version of the Slepian-Fernique inequality that allows the flexibility of using the Dudley's entropy integral for the second metric.

Let $(T,\rho)$ be a metric space. For $A\subset T$, the {\it diameter} of $A$ is $\mathrm{diam}(A,\rho):=\sup_{x,y\in A}\rho(x,y)$. If $\mathcal{A}$ is a partition of $T$ and $t\in T$, then $\mathcal{A}(t)$ denotes the block in $\mathcal{A}$ containing $t$. An {\it admissible} sequence of $T$ is an increasing sequence $\mathcal{A}_{0},\mathcal{A}_{1},\ldots$ of partitions of $T$ such that $\mathcal{A}_{0}=\{T\}$ and $|\mathcal{A}_{n}|\leq 2^{2^{n}}$ for all $n\geq 1$. (By {\it increasing}, we mean that every block of the partition $\mathcal{A}_{n+1}$ is contained in a block of the partition $\mathcal{A}_{n}$.) Define
\[\gamma_{2}(T,\rho):=\inf\sup_{t\in T}\sum_{n\geq 0}2^{\frac{n}{2}}\mathrm{diam}(\mathcal{A}_{n}(t),\rho),\]
where the infimum is over all admissible sequences $\mathcal{A}_{0},\mathcal{A}_{1},\ldots$ of $T$ \cite[Definition 2.7.3]{Talagrandul}.

An {\it ultrametric} $\rho$ on a set $T$ is a metric on $T$ such that $\rho(t_{1},t_{3})\leq\max(\rho(t_{1},t_{2}),\rho(t_{2},t_{3}))$ for all $t_{1},t_{2},t_{3}\in T$. The condition that $\rho(t,s)=0$ only when $t=s$ is not required.
\begin{lemma}[\cite{Talagrandul}, Theorem 16.8.15]\label{ultralemma}
Let $\rho$ be a metric on a set $T$. Let $\delta>0$. Then there exists an ultrametric $\widehat{\rho}$ on $T$ such that $\rho(t,s)\leq\widehat{\rho}(t,s)$, for all $t,s\in T$, and
\[\gamma_{2}(T,\widehat{\rho})\leq(1+\delta)\gamma_{2}(T,\rho).\]
\end{lemma}
\begin{proof}
By the definition of $\gamma_{2}(T,\rho)$, there is an admissible sequence $\mathcal{A}_{0},\mathcal{A}_{1},\ldots$ of $T$ such that
\begin{equation}\label{gamma2almost}
\sup_{t\in T}\sum_{n\geq 0}2^{\frac{n}{2}}\mathrm{diam}(\mathcal{A}_{n}(t),\rho)\leq(1+\delta)\gamma_{2}(T,\rho).
\end{equation}
Let $\mathcal{A}_{\infty}$ be the partition generated by all the $\mathcal{A}_{0},\mathcal{A}_{1},\ldots$.

For $t,s\in T$, define 
\[n(t,s)=\sup\{n\in\mathbb{N}\cup\{0\}:\,t,s\text{ are in the same block of the partition }\mathcal{A}_{n}\}.\]

{\bf Claim 1:} $t,s\in T$ are always in the same block of the partition $\mathcal{A}_{n(t,s)}$.

The claim is trivial when $n(t,s)<\infty$. If $n(t,s)=\infty$ then $t,s$ are in the same block of $\mathcal{A}_{n}$, for all $n\geq 0$, and so $t,s\in\mathcal{A}_{\infty}$ by the definition of $\mathcal{A}_{\infty}$.

{\bf Claim 2:} If $t_{1},t_{2},t_{3}\in T$ then $n(t_{1},t_{3})\geq\min(n(t_{1},t_{2}),n(t_{2},t_{3}))$.

By Claim 1, we have that $t_{1},t_{2}$ are in the same block of the partition $\mathcal{A}_{n(t_{1},t_{2})}$ and thus in the same block of the partition $\mathcal{A}_{\min(n(t_{1},t_{2}),n(t_{2},t_{3}))}$. Similarly, $t_{2},t_{3}$ are in the same block of the partition $\mathcal{A}_{n(t_{2},t_{3})}$ and thus in the same block of $\mathcal{A}_{\min(n(t_{1},t_{2}),n(t_{2},t_{3}))}$. Therefore, $t_{1},t_{2},t_{3}$ are in the same block of $\mathcal{A}_{\min(n(t_{1},t_{2}),n(t_{2},t_{3}))}$. Hence by the definition of $n(t_{1},t_{3})$, we have $n(t_{1},t_{3})\geq\min(n(t_{1},t_{2}),n(t_{2},t_{3}))$.

Define the ultrametric $\widehat{\rho}$ on $T$ as follows:
\[\widehat{\rho}(t,s)=\begin{cases}\mathrm{diam}(\mathcal{A}_{n(t,s)}(t),\rho),&t\neq s\\0,&t=s\end{cases},\]
for $t,s\in T$. Note that since $t,s$ are in the same block of the partition $\mathcal{A}_{n(t,s)}$ (by Claim 1), we have $\mathcal{A}_{n(t,s)}(t)=\mathcal{A}_{n(t,s)}(s)$ and so $\widehat{\rho}(t,s)=\widehat{\rho}(s,t)$ for all $t,s\in T$.

Let us check that $\widehat{\rho}$ is indeed an ultrametric. Fix distinct $t_{1},t_{2},t_{3}\in T$. If $n(t_{1},t_{2})\leq n(t_{2},t_{3})$, then Claim 2 gives $n(t_{1},t_{3})\geq n(t_{1},t_{2})$, and so $\mathcal{A}_{n(t_{1},t_{3})}(t_{1})\subset\mathcal{A}_{n(t_{1},t_{2})}(t_{1})$, which implies that $\widehat{\rho}(t_{1},t_{3})\leq\widehat{\rho}(t_{1},t_{2})$. Similarly, if $n(t_{1},t_{2})\geq n(t_{2},t_{3})$, then Claim 2 gives $n(t_{1},t_{3})\geq n(t_{2},t_{3})$, and so $\mathcal{A}_{n(t_{1},t_{3})}(t_{3})\subset\mathcal{A}_{n(t_{2},t_{3})}(t_{3})$, which implies that $\widehat{\rho}(t_{3},t_{1})\leq\widehat{\rho}(t_{3},t_{2})$. In both cases, we have $\widehat{\rho}(t_{1},t_{3})\leq\max(\widehat{\rho}(t_{1},t_{2}),\widehat{\rho}(t_{2},t_{3}))$ for all distinct $t_{1},t_{2},t_{3}\in T$. If any two of $t_{1},t_{2},t_{3}$ coincide, then we also have $\widehat{\rho}(t_{1},t_{3})\leq\max(\widehat{\rho}(t_{1},t_{2}),\widehat{\rho}(t_{2},t_{3}))$. Therefore, $\widehat{\rho}$ is an ultrametric on $T$.

Fix $t\neq s$ in $T$. By Claim 1, we have that $t,s$ are in the same block of the partition $\mathcal{A}_{n(t,s)}$ and so $\rho(t,s)\leq\mathrm{diam}(\mathcal{A}_{n(t,s)}(t),\rho)=\widehat{\rho}(t,s)$. If $t=s$ then $\rho(t,s)=0=\widehat{\rho}(t,s)$. Therefore, $\rho(t,s)\leq\widehat{\rho}(t,s)$ for all $t,s\in T$.

Finally, for every $j\geq 0$ and every block $A\in\mathcal{A}_{j}$, we have $n(t,s)\geq j$ for all $t,s\in A$, and so $\widehat{\rho}(t,s)\leq\mathrm{diam}(\mathcal{A}_{j}(t),\rho)=\mathrm{diam}(A,\rho)$ for all $t,s\in A$. Hence $\mathrm{diam}(A,\widehat{\rho})\leq\mathrm{diam}(A,\rho)$ for all $A\in\mathcal{A}_{j}$ and $j\geq 0$. So
\[\gamma_{2}(T,\widehat{\rho})\leq\sup_{t\in T}\sum_{n\geq 0}2^{\frac{n}{2}}\mathrm{diam}(\mathcal{A}_{n}(t),\widehat{\rho})\leq\sup_{t\in T}\sum_{n\geq 0}2^{\frac{n}{2}}\mathrm{diam}(\mathcal{A}_{n}(t),\rho).\]
By (\ref{gamma2almost}), the result follows.
\end{proof}
\begin{lemma}\label{sfplus}
Let $\rho$ be a metric on a set $T$. Let $(Z_{t})_{t\in T}$ and $(W_{t})_{t\in T}$ be two Gaussian processes on $T$ such that $\mathbb{E}Z_{t}=\mathbb{E}W_{t}=0$, for all $t\in T$, and
\[\mathbb{E}(Z_{t}-Z_{s})^{2}\leq\mathbb{E}(W_{t}-W_{s})^{2}+\rho(t,s)^{2},\]
for all $t,s\in T$. Then
\[\mathbb{E}\sup_{t\in T}Z_{t}\leq \mathbb{E}\sup_{t\in T}W_{t}+C\gamma_{2}(T,\rho),\footnote{To avoid the measurability issue, for a Gaussian process $(Z_{t})_{t\in T}$, we define $\mathbb{E}\sup_{t\in T}Z_{t}:=\sup_{F}\mathbb{E}\sup_{t\in F}Z_{t}$,
where the supremum is over all finite subset $F\subset T$. See \cite[(2.2)]{Talagrandul}.}\]
where $C\geq 1$ is a universal constant.
\end{lemma}
\begin{proof}
Without loss of generality, $T$ is finite. 
By Lemma \ref{ultralemma}, it suffices to prove the result assuming that $\rho$ is an ultrametric. It is well known that every finite ultrametric space embeds isometrically into a Hilbert space \cite{Bartal}. Thus, there exists $\phi:T\to\mathbb{R}^{|T|}$ such that $\|\phi(t)-\phi(s)\|_{2}=\rho(t,s)$ for all $t,s\in T$. Let $g$ be a random vector in $\mathbb{R}^{|T|}$ with independent standard Gaussian entries and such that $g$ is independent of $(W_{t})_{t\in T}$. Then
\begin{equation}\label{sfpluseq1}
\rho(t,s)^{2}=\|\phi(t)-\phi(s)\|_{2}^{2}=\mathbb{E}\langle g,\phi(t)-\phi(s)\rangle^{2},
\end{equation}
for all $t,s\in T$. So by assumption, we have
\begin{eqnarray*}
\mathbb{E}(Z_{t}-Z_{s})^{2}&\leq&\mathbb{E}(W_{t}-W_{s})^{2}+\mathbb{E}\langle g,\phi(t)-\phi(s)\rangle^{2}\\&=&
\mathbb{E}[W_{t}-W_{s}+\langle g,\phi(t)-\phi(s)\rangle]^{2},
\end{eqnarray*}
for all $t,s\in T$. Consider the Gaussian process $Y_{t}=W_{t}+\langle g,\phi(t)\rangle$ for $t\in T$. Since $\mathbb{E}(Z_{t}-Z_{s})^{2}\leq\mathbb{E}(Y_{t}-Y_{s})^{2}$ and $\mathbb{E}Z_{t}=\mathbb{E}Y_{t}=0$ for all $t,s\in T$, by the Slepian-Fernique inequality \cite[Theorem 7.2.11]{Romanbook}, we have
\[\mathbb{E}\sup_{t\in T}Z_{t}\leq\mathbb{E}\sup_{t\in T}Y_{t}\leq\mathbb{E}\sup_{t\in T}W_{t}+\mathbb{E}\sup_{t\in T}\langle g,\phi(t)\rangle\leq\mathbb{E}\sup_{t\in T}W_{t}+C\gamma_{2}(T,\rho),\]
where the last inequality follows from (\ref{sfpluseq1}) and the generic chaining bound \cite[Theorem 2.7.2]{Talagrandul}.
\end{proof}
\begin{remark}
If we allow a constant factor on the term $\mathbb{E}\sup_{t\in T}W_{t}$ in Lemma \ref{sfplus}, the result can be easily proved using Talagrand's majorizing measure theorem.
\end{remark}
The following result follows immediately from Lemma \ref{sfplus} and the fact that $\gamma_{2}(T,\rho)\leq C\int_{0}^{\infty}\sqrt{\ln N(T,\rho,\epsilon)}\,d\epsilon$ for some universal constant $C\geq 1$ \cite[(2.40)]{Talagrandul}.
\begin{lemma}\label{sfpluscorollary}
Let $\rho$ be a metric on a set $T$. Let $(Z_{t})_{t\in T}$ and $(W_{t})_{t\in T}$ be two Gaussian processes on $T$ such that $\mathbb{E}Z_{t}=\mathbb{E}W_{t}=0$, for all $t\in T$, and
\[\mathbb{E}(Z_{t}-Z_{s})^{2}\leq\mathbb{E}(W_{t}-W_{s})^{2}+\rho(t,s)^{2},\]
for all $t,s\in T$. Then
\[\mathbb{E}\sup_{t\in T}Z_{t}\leq \mathbb{E}\sup_{t\in T}W_{t}+C\int_{0}^{\infty}\sqrt{\ln N(T,\rho,\epsilon)}\,d\epsilon,\]
where $C\geq 1$ is a universal constant.
\end{lemma}
\begin{proof}[Proof of Theorem \ref{main}]
{\bf Step 1: Gaussian process}\\
The injective norm of the random tensor $Z$ defined in Theorem \ref{main} is equal to
\[\|Z\|_{\mathrm{inj}}=\sup_{x_{1},\ldots,x_{r}\in B_{2}^{d}}\sum_{i_{1},\ldots,i_{r}\in[d]}b_{i_{1},\ldots,i_{r}}g_{i_{1},\ldots,i_{r}}\langle x_{1},e_{i_{1}}\rangle\ldots\langle x_{r},e_{i_{r}}\rangle.\]
Let $B_{2}^{d,r}:=\underbrace{B_{2}^{d}\times\ldots\times B_{2}^{d}}_{r}$. Define the Gaussian process $(Z_{(x_{1},\ldots,x_{r})})_{(x_{1},\ldots,x_{r})\in B_{2}^{d,r}}$ by
\[Z_{(x_{1},\ldots,x_{r})}:=\sum_{i_{1},\ldots,i_{r}\in[d]}b_{i_{1},\ldots,i_{r}}g_{i_{1},\ldots,i_{r}}\langle x_{1},e_{i_{1}}\rangle\ldots\langle x_{r},e_{i_{r}}\rangle,\]
for $(x_{1},\ldots,x_{r})\in B_{2}^{d,r}$. Then
\begin{equation}\label{proofmaineq1}
\mathbb{E}\|Z\|_{\mathrm{inj}}=\mathbb{E}\sup_{(x_{1},\ldots,x_{r})\in B_{2}^{d,r}}Z_{(x_{1},\ldots,x_{r})}.
\end{equation}

{\bf Step 2: Bound the induced metric}\\
In this step, we use the properties listed in Subsection \ref{summary} to bound the induced metric of the Gaussian process $(Z_{(x_{1},\ldots,x_{r})})_{(x_{1},\ldots,x_{r})\in B_{2}^{d,r}}$.

For all $(x_{1},\ldots,x_{r}),(y_{1},\ldots,y_{r})\in B_{2}^{d,r}$, we have
\begin{align*}
&(\mathbb{E}|Z_{(x_{1},\ldots,x_{r})}-Z_{(y_{1},\ldots,y_{r})}|^{2})^{\frac{1}{2}}\\=&
\left(\sum_{i_{1},\ldots,i_{r}\in[d]}b_{i_{1},\ldots,i_{r}}^{2}(\langle x_{1},e_{i_{1}}\rangle\ldots\langle x_{r},e_{i_{r}}\rangle-\langle y_{1},e_{i_{1}}\rangle\ldots\langle y_{r},e_{i_{r}}\rangle)^{2}\right)^{\frac{1}{2}}\\=&
\|\tau(x_{1},\ldots,x_{r})-\tau(y_{1},\ldots,y_{r})\|_{2}\quad\text{by }(\ref{staudef})\\=&
\left\|\sum_{k=1}^{r}\tau(y_{1},\ldots,y_{k-1},x_{k}-y_{k},x_{k+1},\ldots,x_{r})\right\|_{2}\\\le&
\sum_{k=1}^{r}\|\tau(y_{1},\ldots,y_{k-1},x_{k}-y_{k},x_{k+1},\ldots,x_{r})\|_{2}\\=&
\sum_{k=1}^{r}\|D_{y_{1},\ldots,y_{k-1},x_{k+1},\ldots,x_{r}}^{(k)}(x_{k}-y_{k})\|_{2}\quad\text{by }(\ref{staud})\\\leq&
\sum_{k=1}^{r}\|D_{x_{1},\ldots,x_{k-1},x_{k+1},\ldots,x_{r}}^{(k)}x_{k}-D_{y_{1},\ldots,y_{k-1},y_{k+1},\ldots,y_{r}}^{(k)}y_{k}\|_{2}\\&
+\sum_{k=1}^{r}
\|D_{x_{1},\ldots,x_{k-1},x_{k+1},\ldots,x_{r}}^{(k)}-D_{y_{1},\ldots,y_{k-1},x_{k+1},\ldots,x_{r}}^{(k)}\|_{\infty}\\&+
\sum_{k-1}^{r}\|D_{y_{1},\ldots,y_{k-1},y_{k+1},\ldots,y_{r}}^{(k)}-D_{y_{1},\ldots,y_{k-1},x_{k+1},\ldots,x_{r}}^{(k)}\|_{\infty},
\end{align*}
where the last inequality follows from the triangle inequality and the facts that $\|x_{k}\|_{2},\|y_{k}\|_{2}\leq 1$ and all the $D$'s are diagonal matrices. Define the metric $\eta$ on $B_{2}^{d,r}$ by
\begin{equation}\label{etadef}
\eta((x_{1},\ldots,x_{r}),(y_{1},\ldots,y_{r})):=\eta^{(1)}(x_{1},y_{1})+\ldots+\eta^{(r)}(x_{r},y_{r}),
\end{equation}
for $(x_{1},\ldots,x_{r}),(y_{1},\ldots,y_{r})\in B_{2}^{d,r}$. Then by (\ref{sdeta}), it follows that
\begin{align*}
&(\mathbb{E}|Z_{(x_{1},\ldots,x_{r})}-Z_{(y_{1},\ldots,y_{r})}|^{2})^{\frac{1}{2}}\\=&
\sum_{k=1}^{r}\|D_{x_{1},\ldots,x_{k-1},x_{k+1},\ldots,x_{r}}^{(k)}x_{k}-D_{y_{1},\ldots,y_{k-1},y_{k+1},\ldots,y_{r}}^{(k)}y_{k}\|_{2}\\&+
\sum_{k=1}^{r}(\eta^{(1)}(x_{1},y_{1})+\ldots+\eta^{(k-1)}(x_{k-1},y_{k-1}))\\&+
\sum_{k=1}^{r}(\eta^{(k+1)}(y_{k+1},x_{k+1})+\ldots+\eta^{(r)}(y_{r},x_{r}))\\\leq&
\sum_{k=1}^{r}\|D_{x_{1},\ldots,x_{k-1},x_{k+1},\ldots,x_{r}}^{(k)}x_{k}-D_{y_{1},\ldots,y_{k-1},y_{k+1},\ldots,y_{r}}^{(k)}y_{k}\|_{2}\\&
+r\cdot\eta((x_{1},\ldots,x_{r}),(y_{1},\ldots,y_{r})).
\end{align*}
So taking square and applying the inequality $(a+b)^{2}\leq 2a^{2}+2b^{2}$, we obtain
\begin{align*}
&\mathbb{E}(Z_{(x_{1},\ldots,x_{r})}-Z_{(y_{1},\ldots,y_{r})})^{2}\\\leq&
2r\sum_{k=1}^{r}\|D_{x_{1},\ldots,x_{k-1},x_{k+1},\ldots,x_{r}}^{(k)}x_{k}-D_{y_{1},\ldots,y_{k-1},y_{k+1},\ldots,y_{r}}^{(k)}y_{k}\|_{2}^{2}\\&
+2r^{2}\cdot\eta((x_{1},\ldots,x_{r}),(y_{1},\ldots,y_{r}))^{2}.
\end{align*}

{\bf Step 3: Apply Lemma \ref{sfpluscorollary}}\\
Let $g^{(1)},\ldots,g^{(r)}$ be independent random vectors in $\mathbb{R}^{d}$ and such that each $g^{(k)}$ has independent standard Gaussian entries. Define the following Gaussian process on $B_{2}^{d,r}$
\[W_{(x_{1},\ldots,x_{r})}:=\sum_{k=1}^{r}\langle g^{(k)},D_{x_{1},\ldots,x_{k-1},x_{k+1},\ldots,x_{r}}^{(k)}x_{k}\rangle,\]
for $(x_{1},\ldots,x_{r})\in B_{2}^{d,r}$. Then
\begin{align*}
&\mathbb{E}(Z_{(x_{1},\ldots,x_{r})}-Z_{(y_{1},\ldots,y_{r})})^{2}\\\leq&
2r\mathbb{E}(W_{(x_{1},\ldots,x_{r})}-W_{(y_{1},\ldots,y_{r})})^{2}
+2r^{2}\cdot\eta((x_{1},\ldots,x_{r}),(y_{1},\ldots,y_{r}))^{2}.
\end{align*}
So by Lemma \ref{sfpluscorollary},
\begin{eqnarray*}
\mathbb{E}\sup_{(x_{1},\ldots,x_{r})\in B_{2}^{d,r}}Z_{(x_{1},\ldots,x_{r})}&\leq&\mathbb{E}\sup_{(x_{1},\ldots,x_{r})\in B_{2}^{d,r}}\sqrt{2r}W_{(x_{1},\ldots,x_{r})}\\&&
+C\int_{0}^{\infty}\sqrt{\ln N(B_{2}^{d,r},\sqrt{2r^{2}}\cdot\eta,\epsilon)}\,d\epsilon\\&=&
\sqrt{2r}\cdot\mathbb{E}\sup_{x_{1},\ldots,x_{r}\in B_{2}^{d}}\sum_{k=1}^{r}\langle g^{(k)},D_{x_{1},\ldots,x_{k-1},x_{k+1},\ldots,x_{r}}^{(k)}x_{k}\rangle\\&&
+C\sqrt{2r^{2}}\int_{0}^{\infty}\sqrt{\ln N(B_{2}^{d,r},\eta,\epsilon)}\,d\epsilon.
\end{eqnarray*}
Thus, by (\ref{proofmaineq1}) and applying `$\sup\sum\leq\sum\sup$," we have
\begin{eqnarray}\label{proofmaineq2}
\mathbb{E}\|Z\|_{\mathrm{inj}}&\leq&
\sqrt{2r}\sum_{k=1}^{r}\mathbb{E}\sup_{x_{1},\ldots,x_{r}\in B_{2}^{d}}\langle g,D_{x_{1},\ldots,x_{k-1},x_{k+1},\ldots,x_{r}}^{(k)}x_{k}\rangle\\&&
+Cr\sqrt{2}\int_{0}^{\infty}\sqrt{\ln N(B_{2}^{d,r},\eta,\epsilon)}\,d\epsilon,\nonumber
\end{eqnarray}
where $g$ is a random vector in $\mathbb{R}^{d}$ with independent standard Gaussian entries.

{\bf Step 4: Bound the two terms}\\
In this final step, we bound the two terms in the upper bound for $\mathbb{E}\|Z\|_{\mathrm{inj}}$ in (\ref{proofmaineq2}). For the first term, for all $x_{1},\ldots,x_{r}\in B_{2}^{d}$,
\begin{eqnarray*}
\langle g,D_{x_{1},\ldots,x_{k-1},x_{k+1},\ldots,x_{r}}^{(k)}x_{k}\rangle&\leq&
\|D_{x_{1},\ldots,x_{k-1},x_{k+1},\ldots,x_{r}}^{(k)}g\|_{2}\\&=&
\|\tau(x_{1},\ldots,x_{k-1},g,x_{k+1},\ldots,x_{r})\|_{2}\quad\text{by }(\ref{staud})\\&=&
\left(\sum_{i_{1},\ldots,i_{r}\in[d]}b_{i_{1},\ldots,i_{r}}^{2}\langle g,e_{i_{k}}\rangle^{2}\prod_{j\in[r]\backslash\{k\}}\langle x_{j},e_{i_{j}}\rangle^{2}\right)^{\frac{1}{2}}\quad\text{by }(\ref{staudef}),
\end{eqnarray*}
and so
\begin{equation}\label{proofmaineq3}
\sup_{x_{1},\ldots,x_{r}\in B_{2}^{d}}\langle g,D_{x_{1},\ldots,x_{k-1},x_{k+1},\ldots,x_{r}}^{(k)}x_{k}\rangle\leq\max_{i_{1},\ldots,i_{k-1},i_{k+1},\ldots,i_{r}\in[d]}\left(\sum_{i_{k}\in [d]}b_{i_{1},\ldots,i_{r}}^{2}\langle g,e_{i_{k}}\rangle^{2}\right)^{\frac{1}{2}}.
\end{equation}
Let $b=\max_{i_{1},\ldots,i_{r}\in[d]}|b_{i_{1},\ldots,i_{r}}|$. By Gaussian concentration \cite[Equation (2.35)]{Ledoux}, for all $t\geq 0$ and $i_{1},\ldots,i_{k-1},i_{k+1},\ldots,i_{r}\in[d]$, with probability at least $1-e^{-t^{2}/(2b^{2})}$, we have
\[\left(\sum_{i_{k}\in [d]}b_{i_{1},\ldots,i_{r}}^{2}\langle g,e_{i_{k}}\rangle^{2}\right)^{\frac{1}{2}}\leq
\mathbb{E}\left(\sum_{i_{k}\in [d]}b_{i_{1},\ldots,i_{r}}^{2}\langle g,e_{i_{k}}\rangle^{2}\right)^{\frac{1}{2}}+t\leq
\left(\sum_{i_{k}\in [d]}b_{i_{1},\ldots,i_{r}}^{2}\right)^{\frac{1}{2}}+t.\]
So by a union bound, for every $t\geq 0$, with probability at least $1-d^{r-1}e^{-t^{2}/(2b^{2})}$, we have
\[\max_{i_{1},\ldots,i_{k-1},i_{k+1},\ldots,i_{r}\in[d]}\left(\sum_{i_{k}\in [d]}b_{i_{1},\ldots,i_{r}}^{2}\langle g,e_{i_{k}}\rangle^{2}\right)^{\frac{1}{2}}\leq\max_{i_{1},\ldots,i_{k-1},i_{k+1},\ldots,i_{r}\in[d]}\mathbb{E}\left(\sum_{i_{k}\in [d]}b_{i_{1},\ldots,i_{r}}^{2}\right)^{\frac{1}{2}}+t.\]
Thus,
\begin{align*}
&\mathbb{E}\max_{i_{1},\ldots,i_{k-1},i_{k+1},\ldots,i_{r}\in[d]}\left(\sum_{i_{k}\in [d]}b_{i_{1},\ldots,i_{r}}^{2}\langle g,e_{i_{k}}\rangle^{2}\right)^{\frac{1}{2}}\\\leq&
\max_{i_{1},\ldots,i_{k-1},i_{k+1},\ldots,i_{r}\in[d]}\mathbb{E}\left(\sum_{i_{k}\in [d]}b_{i_{1},\ldots,i_{r}}^{2}\right)^{\frac{1}{2}}+\int_{0}^{\infty}\min(d^{r-1}e^{-t^{2}/(2b^{2})},1)\,dt\\\leq&
\max_{i_{1},\ldots,i_{k-1},i_{k+1},\ldots,i_{r}\in[d]}\mathbb{E}\left(\sum_{i_{k}\in [d]}b_{i_{1},\ldots,i_{r}}^{2}\right)^{\frac{1}{2}}+Cb\sqrt{r\ln d}.
\end{align*}
So by (\ref{proofmaineq3}), we have
\begin{align}\label{proofmaineq4}
&\mathbb{E}\sup_{x_{1},\ldots,x_{r}\in B_{2}^{d}}\langle g,D_{x_{1},\ldots,x_{k-1},x_{k+1},\ldots,x_{r}}^{(k)}x_{k}\rangle\\\leq&
\max_{i_{1},\ldots,i_{k-1},i_{k+1},\ldots,i_{r}\in[d]}\mathbb{E}\left(\sum_{i_{k}\in [d]}b_{i_{1},\ldots,i_{r}}^{2}\right)^{\frac{1}{2}}+C_{1}b\sqrt{r\ln d},\nonumber
\end{align}
for some universal constant $C_{1}\geq 1$.

We now bound the second term in the upper bound in (\ref{proofmaineq2}). By the definition of $\eta$ in (\ref{etadef}), the covering number
\[N(B_{2}^{d,r},\eta,\epsilon)\leq \prod_{k=1}^{r}N\left(B_{2}^{d},\eta^{(k)},\frac{\epsilon}{r}\right).\]
So
\begin{eqnarray*}
\int_{0}^{\infty}\sqrt{\ln N(B_{2}^{d,r},\eta,\epsilon)}\,d\epsilon&\leq&\int_{0}^{\infty}\sqrt{r}\max_{k\in[r]}\sqrt{\ln N\left(B_{2}^{d},\eta^{(k)},\frac{\epsilon}{r}\right)}\,d\epsilon\\&=&
r\sqrt{r}\int_{0}^{\infty}\max_{k\in[r]}\sqrt{\ln N(B_{2}^{d},\eta^{(k)},\epsilon)}\,d\epsilon\leq C_{2}r^{2}(\ln d)^{2}b,
\end{eqnarray*}
for some universal constant $C_{2}\geq 1$, where the last inequality follows from (\ref{setadudley}). Therefore, by (\ref{proofmaineq2}) and (\ref{proofmaineq4}), we have
\begin{eqnarray*}
\mathbb{E}\|Z\|_{\mathrm{inj}}&\leq&\sqrt{2r}\sum_{k=1}^{r}\max_{i_{1},\ldots,i_{k-1},i_{k+1},\ldots,i_{r}\in[d]}\mathbb{E}\left(\sum_{i_{k}\in[d]}b_{i_{1},\ldots,i_{r}}^{2}\right)^{\frac{1}{2}}\\&&
+\sqrt{2}C_{1}br^{2}\sqrt{\ln d}+\sqrt{2}CC_{2}r^{3}(\ln d)^{2}b.
\end{eqnarray*}
Theorem \ref{main} is proved.
\end{proof}

{\bf Acknowledgement:} The author is grateful to Ramon van Handel for some very useful suggestions and to Yizhe Zhu for pointing out Corollary \ref{hypergraph}.

\end{document}